\documentclass{article}

\usepackage{amsmath} 
\usepackage{amsthm}
\usepackage{amssymb}
\usepackage[sort,nocompress,noadjust]{cite}
\usepackage{color}
\usepackage{booktabs}
\usepackage{graphicx}
\usepackage{subcaption}
\usepackage{hyperref}
\usepackage[final]{showlabels}

\newtheorem{theorem}{Theorem}
\newtheorem{lemma}{Lemma}
 
\newtheorem{corollary}[theorem]{Corollary}
\theoremstyle{definition}
\newtheorem*{definition}{Definition}

\DeclareMathOperator{\arb}{arb}
\theoremstyle{remark}

\newcommand{\ep}{\varepsilon}

\title{Further results on arc and bar $k$-visibility graphs}
\author{Mehtaab Sawhney\\
\href{mailto:msawhney98@yahoo.com}{msawhney98@yahoo.com}
\and
Jonathan Weed\\
\href{mailto:jweed@mit.edu}{jweed@mit.edu}}
\date{\today}
\begin{document}

\maketitle

\begin{abstract}We consider visibility graphs involving bars and arcs in which lines of sight can pass through up to $k$ objects. We prove a new edge bound for arc $k$-visibility graphs, provide maximal constructions for arc and semi-arc $k$-visibility graphs, and give a complete characterization of semi-arc visibility graphs. We show that the family of arc $i$-visibility graphs is never contained in the family of bar $j$-visibility graphs for any $i$ and $j$, and that the family of bar $i$-visibility graphs is not contained in the family of bar $j$-visibility graphs for $i \neq j$. We also give the first thickness bounds for arc and semi-arc $k$-visibility graphs. Finally, we introduce a model for random semi-bar and semi-arc $k$-visibility graphs and analyze its properties.
\end{abstract}

\section{Introduction}
Visibility graphs are a general class of graphs that model lines of sight between geometric regions. Studied since at least the 1960's, these graphs have received attention partially due to their utility in modeling topics of practical interest, such as very-large-scale integration (VLSI)~\cite{schlag1983algorithm,duchet1983representing} and robot motion planning~\cite{nilsson1969mobile,oommen1987robot}.

In this work, we primarily consider \emph{arc} visibility graphs, in which we restrict our attention to radial lines of sight between concentric circular arcs. (Precise definitions of these and other concepts appear in Section~\ref{sec:defs}.) These graphs, introduced and characterized by Hutchinson~\cite{hutchinson2002arc} under the name polar visibility graphs, possess a natural connection to the geometry of the projective plane.

In standard visibility graphs, lines of sight are taken to be line segments in the plane that intersected no regions except at their endpoints. However, since the seminal paper of Dean et al.~\cite{dean2007bar}, more general constructions have been considered as well. The notion of \emph{$k$-visibility} they define allows lines of sight to intersect up to $k$ regions in addition to the two regions at the endpoints.

Using this definition, Babbitt et al.~\cite{babbitt2013k} generalized Hutchinson's work by considering arc $k$-visibility graphs and introduced semi-arc $k$-visibility graphs by analogy to the semi-bar visibility graphs defined by Felnsner and Massow~\cite{felsner2008parameters}. They proved edge bounds for these graphs and posed a number of open questions about their other properties.

In this work, we significantly extend the work of Babbitt et al.~\cite{babbitt2013k}. In Section~\ref{sec:class}, we give a complete characterization of semi-arc visibility graphs, which complements the characterization of arc visibility graphs given by Hutchinson~\cite{hutchinson2002arc}. In Section~\ref{sec:edge}, we prove a stronger edge bound for arc $k$-visibility graphs and provide constructions showing edge bounds for arc visibility graphs and for semi-arc $k$-visibility graphs to be tight. In Section~\ref{sec:thick}, we give the first nontrivial thickness bounds for arc and semi-arc $k$-visibility graphs.
In Section~\ref{sec:inclusion}, we consider the relationship between arc $k$-visibility graphs and the more common class of bar $k$-visibility graphs. Finally, in Section~\ref{sec:random}, we introduce and analyze a model for random semi-arc $k$-visibility graphs.

\section{Preliminaries}
\label{sec:defs}
\subsection{Visibility and $k$-visibility graphs}
We begin by briefly defining the notion of a visibility graph before specializing to the particular classes that will be the focus of this work.

A visibility graph is a graph whose vertices correspond to regions in the plane. Two vertices are connected whenever the corresponding regions are connected by an unobstructed line of sight. We can define a class of such graphs by specifying a family of allowable regions and lines of sight. A graph arising in this way is known as a \emph{visibility graph}, and the corresponding arrangement of regions is known as a \emph{visibility representation}.

Dean et al.~\cite{dean2007bar} generalized this concept by defining \emph{$k$-visibility graphs}, which are identical to visibility graphs except that a line of sight may intersect up to $k$ regions in addition to the two regions it connects.

\begin{figure}[h]
\centering
\begin{subfigure}[b]{.40\linewidth}
  \includegraphics[width=\linewidth]{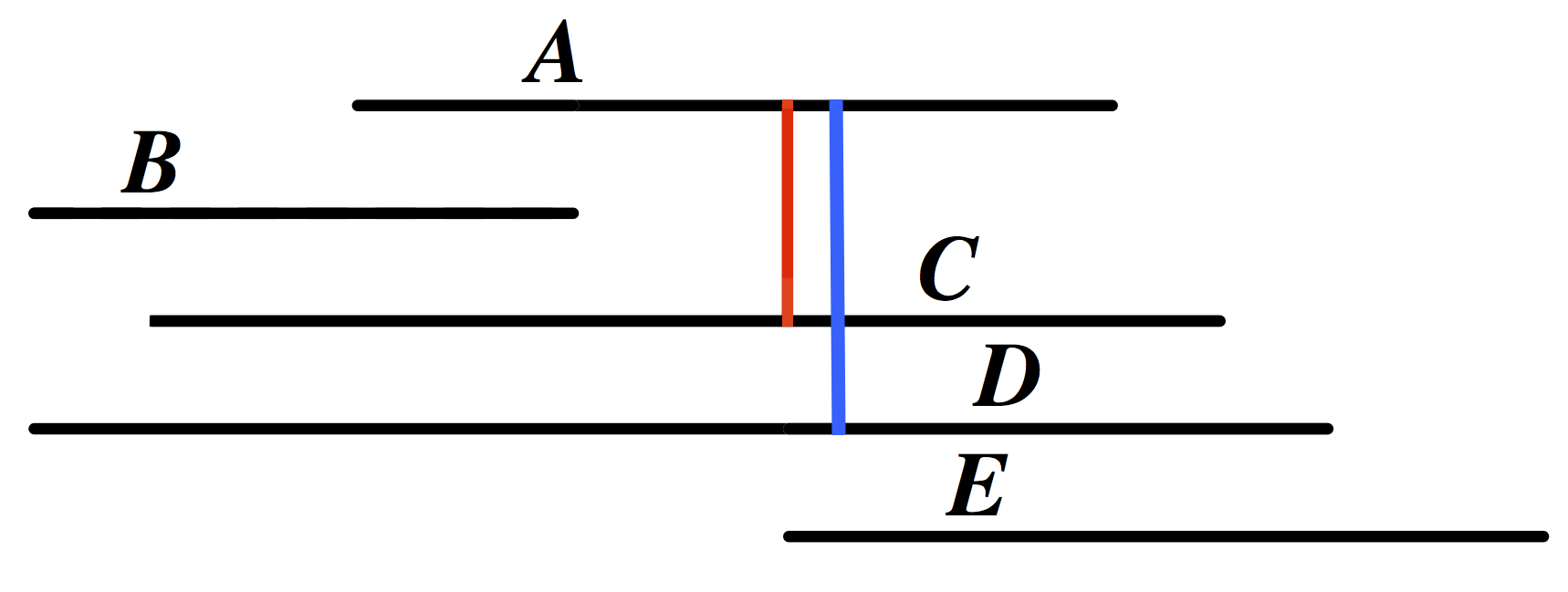}
  \caption{Visibility representation}
\end{subfigure}
\hfill
\begin{subfigure}[b]{.25\linewidth}
  \includegraphics[width=\linewidth]{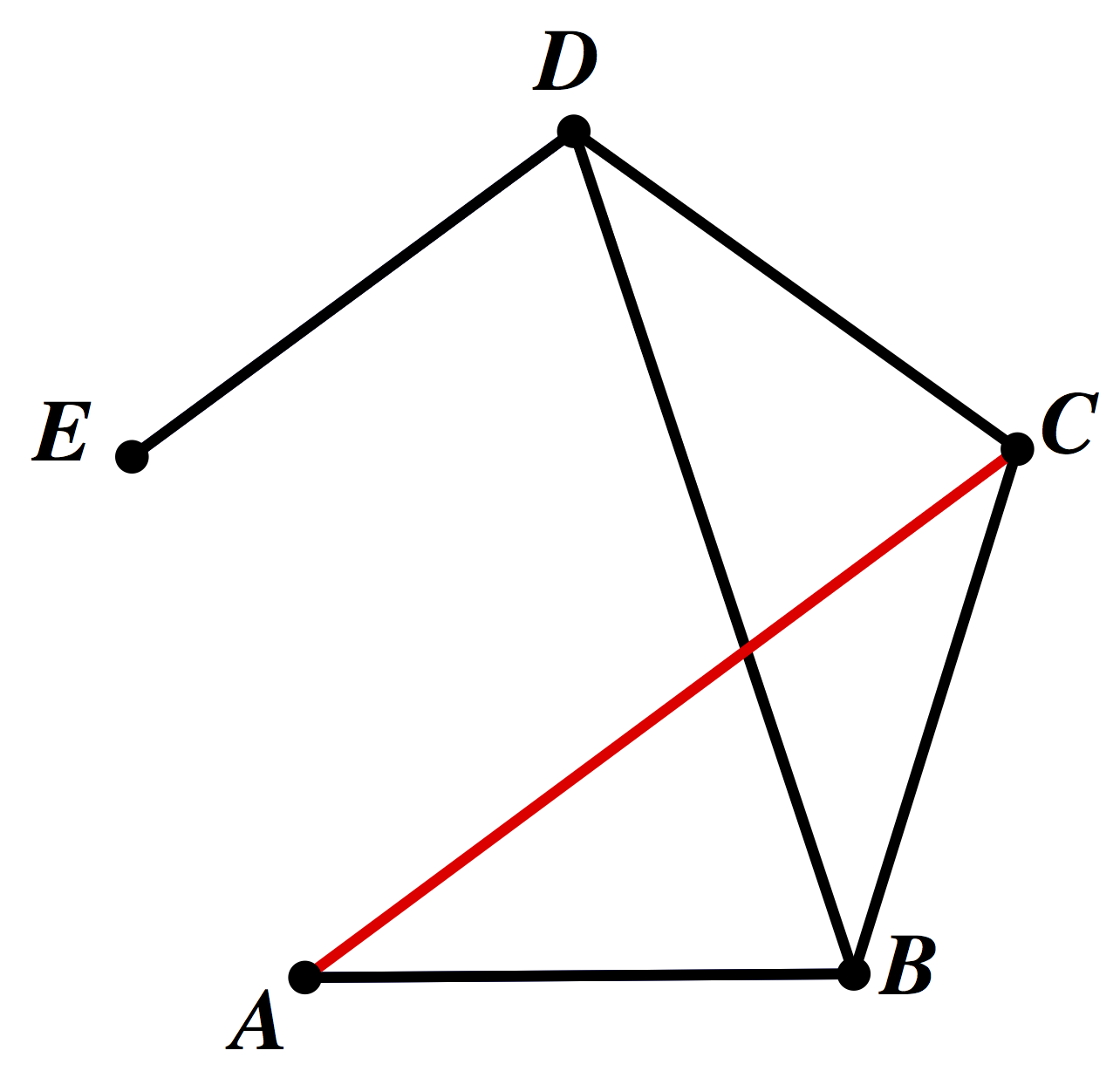}
    \caption{Visibility graph}
\end{subfigure}
\hfill
\begin{subfigure}[b]{.25\linewidth}
  \includegraphics[width=\linewidth]{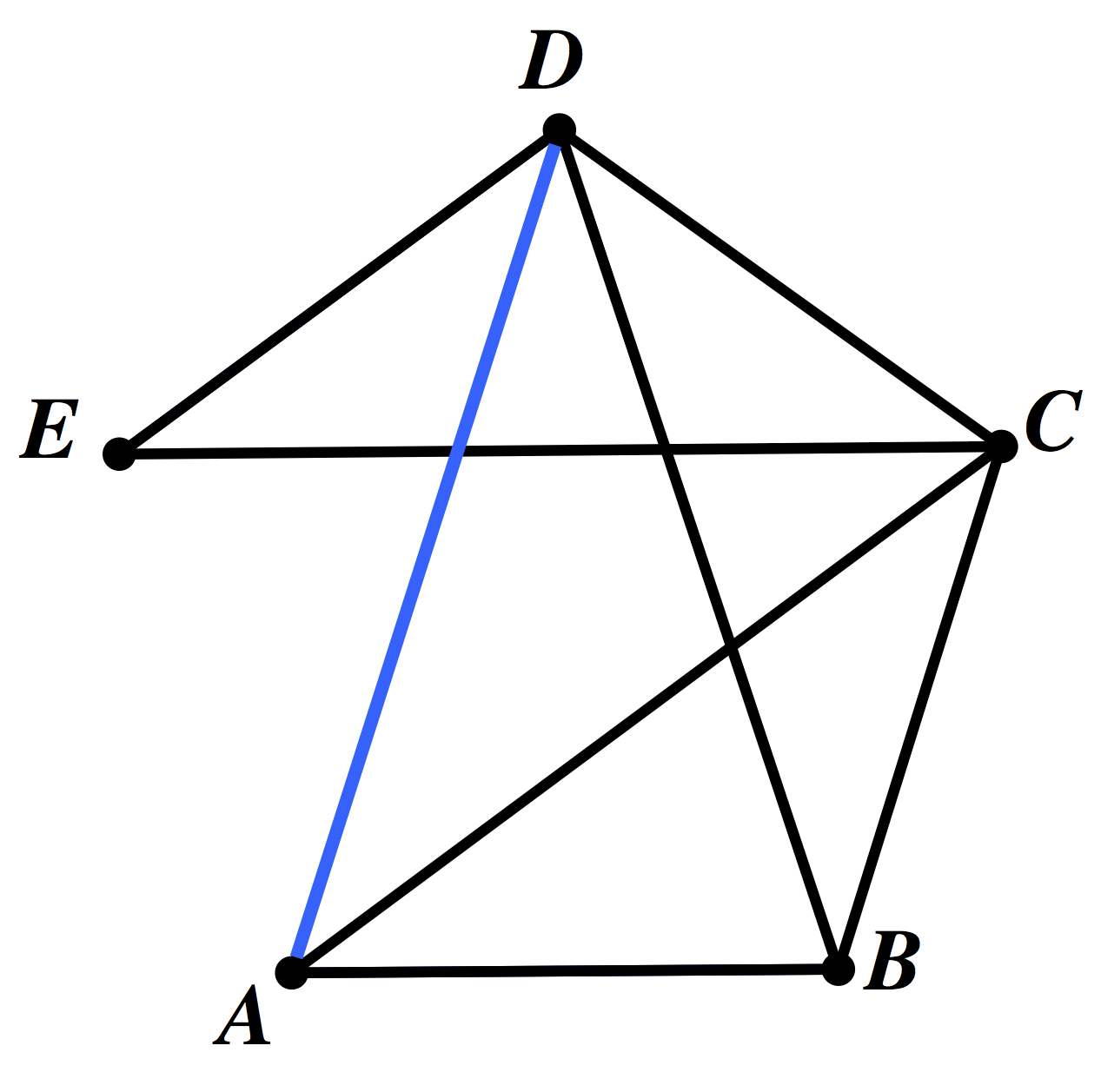}
    \caption{$1$-visibility graph}
\end{subfigure}
\caption{Bar visibility}
\label{fig:barkdef}
\end{figure}

The simplest example is the family of bar ($k$-)visibility graphs. These are defined by taking the regions to be nonintersecting closed horizontal line segments in the plane (``bars'') connected by vertical lines of sight. Requiring lines of sight to be unobstructed yields bar visibility graphs; allowing them to intersect up to $k$ additional bars yields bar $k$-visibility graphs. Figure~\ref{fig:barkdef} shows a collection of bars and the corresponding visibility and $1$-visibility graphs. 

\begin{figure}[h]
\centering
\begin{subfigure}[b]{.3\linewidth}
  \includegraphics[width=.9\linewidth]{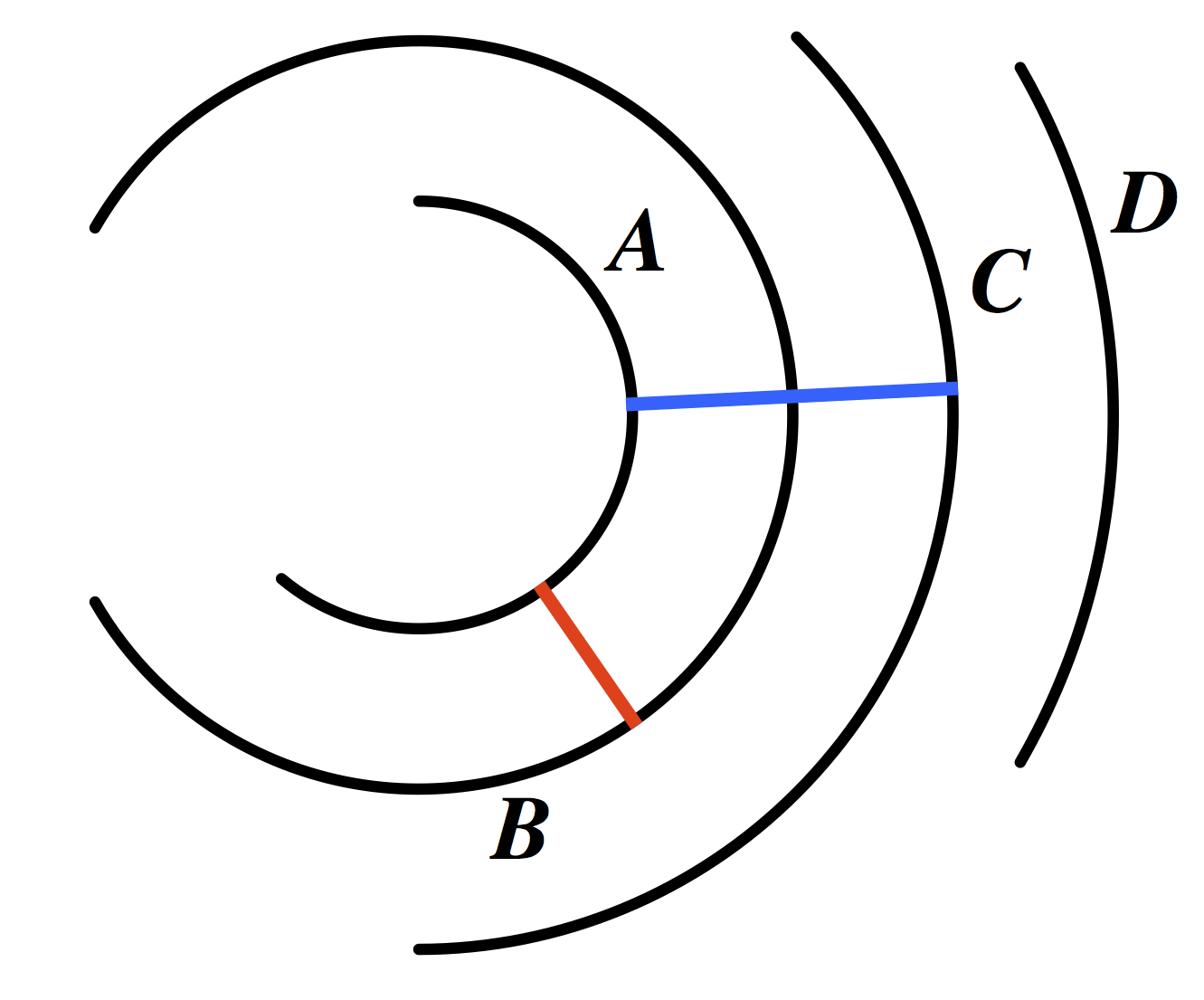}
  \caption{Visibility representation}
\end{subfigure}
\hfill
\begin{subfigure}[b]{.25\linewidth}
  \includegraphics[width=.9\linewidth]{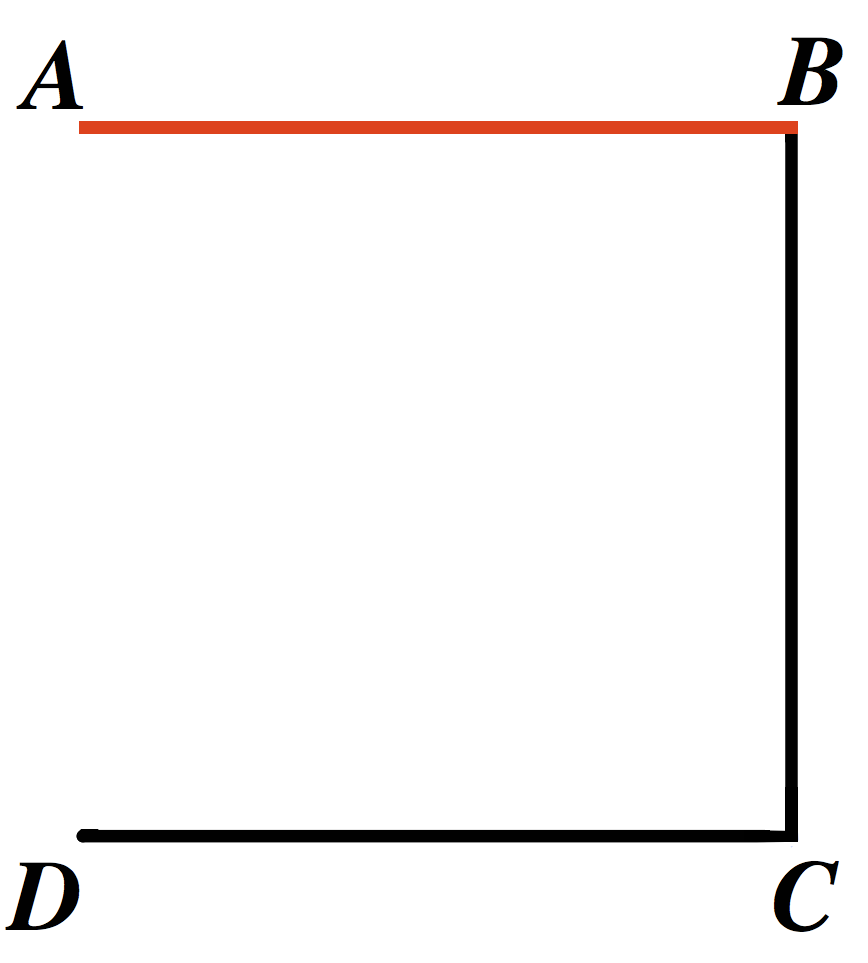}
  \caption{Visibility graph}
\end{subfigure}
\hfill
\begin{subfigure}[b]{.25\linewidth}
  \includegraphics[width=.9\linewidth]{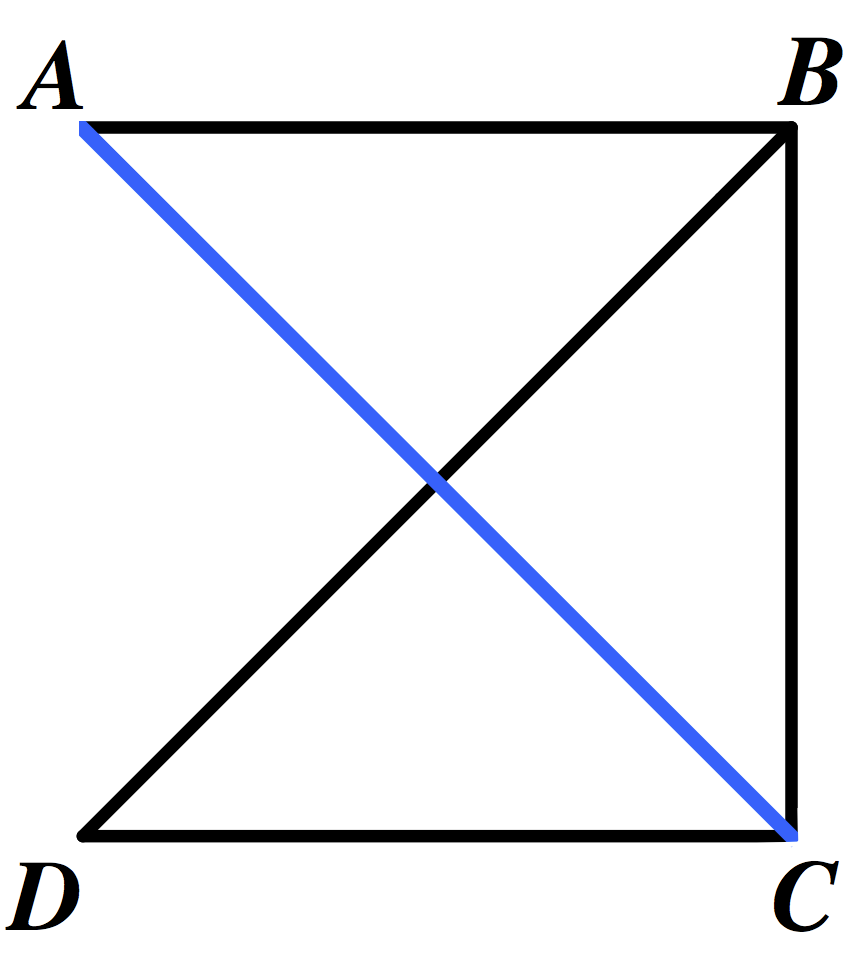}
  \caption{$1$-visibility graph}
\end{subfigure}
\caption{Arc visibility}
\label{fig:arckdef}
\end{figure}

Of primary importance in the sequel are arc ($k$-)visibility graphs, introduced by Hutchinson~\cite{hutchinson2002arc} and~Dean et al.~\cite{dean2007bar}. These are defined by taking the regions to be nonintersecting concentric circular arcs and lines of sight to be radial line segments, which may pass through the center of the circle. As above, visibility and $k$-visibility graphs obtain when lines of sight may pass through $0$ or at most $k$ intervening bars, respectively. Examples appear in Figure~\ref{fig:arckdef}.

There is a slight subtlety in defining arc $k$-visibility graphs, since a radial line may intersect an obtuse arc more than once. We adopt the convention that for the purpose of counting visibilities these double intersections are counted only once.

Finally, we consider in addition two important special cases of the classes defined above. Semi-bar visibility graphs, introduced by Felsner and Massow~\cite{felsner2008parameters}, are bar visibility graphs where we insist that the left endpoints of all the bars lie on the same vertical line. Likewise, in semi-arc visibility graphs, introduced by Babbitt et al.~\cite{babbitt2013k}, arcs extend in a counterclockwise direction from the same radial ray. Figure~\ref{fig:semi} gives examples of semi-bar and semi-arc visibility representations.

\begin{figure}[h]
\centering
\begin{subfigure}[b]{.45\linewidth}
\centering
  \includegraphics[width=.6\linewidth]{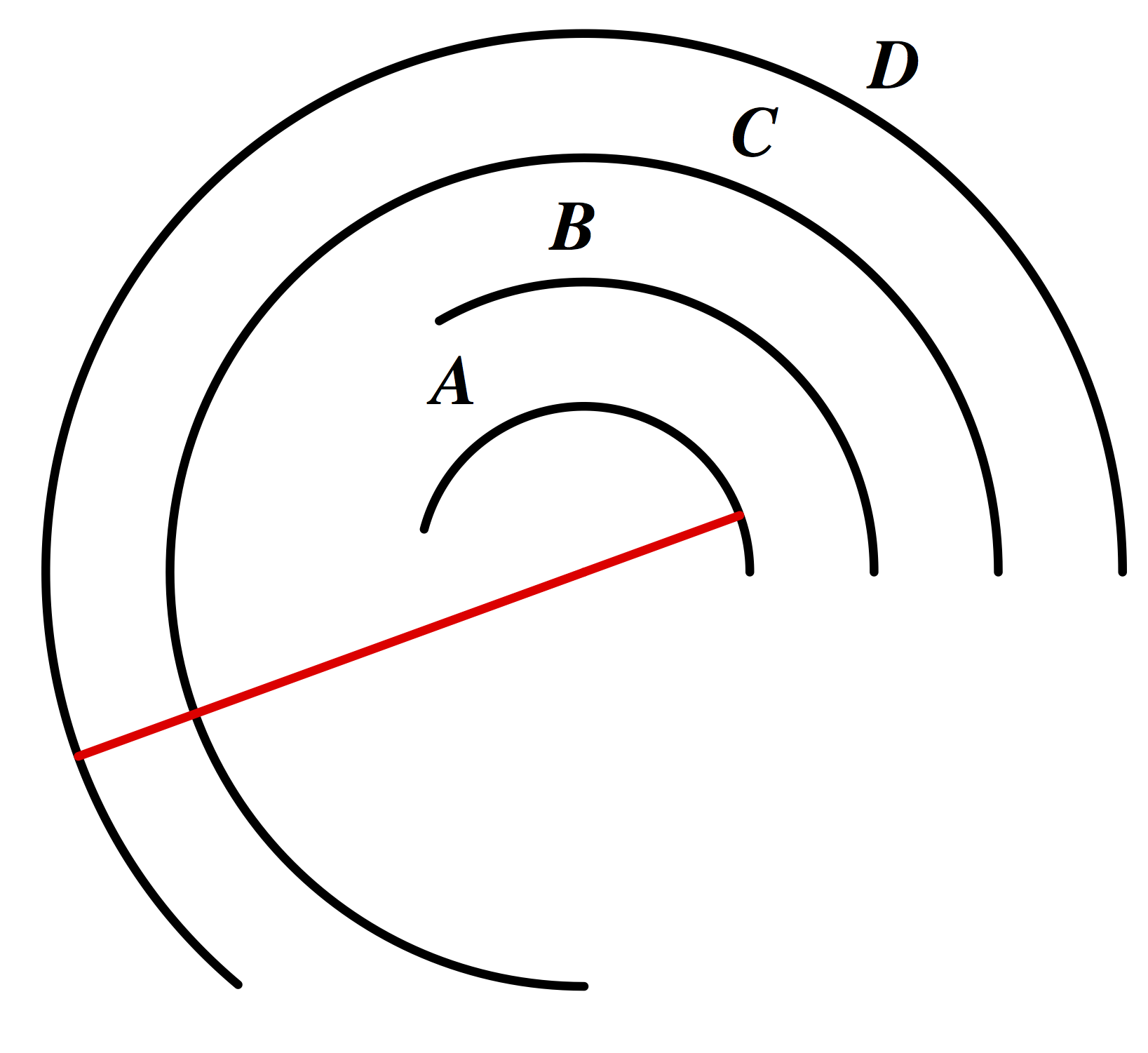}
  \caption{Semi-arc visibility representation}
\end{subfigure}
~
\begin{subfigure}[b]{.45\linewidth}
  \includegraphics[width=\linewidth]{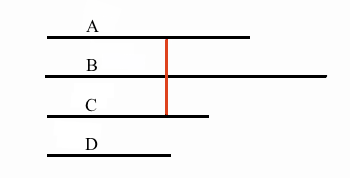}
  \caption{Semi-bar visibility represenation}
\end{subfigure}
\caption{}
\label{fig:semi}
\end{figure}

\subsection{Notation and terminology}
For~$n \geq 3$, denote by~$K_n$ the complete graph on~$n$ vertices and by~$C_n$ the cycle on~$n$ vertices.

The \emph{argument} of a point or ray in an arc or semi-arc visibility representation is its angular position, when measured with respect to the positive $x$-axis. Given an arc in such a representation, it is possible to choose arguments $\alpha$ and $\beta$ for its endpoints such that $0 < \beta - \alpha < 2 \pi$. Call the endpoint corresponding to $\beta$ the \emph{positive} endpoint and the endpoint corresponding to $\alpha$ the \emph{negative} endpoint.

\section{Classification of semi-arc visibility graphs}
\label{sec:class}
Before obtaining a full classification of semi-arc visibility graphs, we first consider the question of planarity. Babbitt et al.~\cite{babbitt2013k} observe that $K_5$ is an arc visibility graph, so not all arc visibility graphs are planar. The following theorem shows that no such example is possible once we restrict to semi-arc visibility graphs.
\begin{theorem}\label{thm:arcplanar}
All semi-arc visibility graphs are planar.
\end{theorem}
\begin{proof}
Fix a semi-arc visibility graph $G$ and a corresponding representation. We can assume that all arcs have radially distinct endpoints and distinct radii, since we can perturb the arcs slightly to yield a graph that is nonplanar if $G$ is.

To show that $G$ is planar, we will alter the representation of $G$ to produce a representation of a new graph $G''$ in such a way that $G$ is planar if $G''$ is. We conclude by producing an explicit planar embedding of $G''$.

Label the arcs $a_1, \dots, a_n$ with indices increasing with increasing radius, and label the vertices of $G$ as $v_1, \dots, v_n$ such that $v_i$ corresponds to $a_i$ for all $i$. We begin by removing arcs that do not affect the planarity of $G$. Suppose there is a sequence $a_{j-1}, a_j, a_{j+1}$ such that the arc $a_j$ has a smaller argument than both $a_{j-1}$ and $a_{j+1}$. Then $v_j$ is connected to only $v_{j-1}$ and $v_{j+1}$ in $G$. Removing $a_j$ from the representation corresponds to contracting the edge between $v_j$ and $v_{j+1}$ (or equivalently the edge between $v_{j-1}$ and $v_j$), which does not affect the planarity of the graph.

So we can assume that we have a semi-arc visibility representation in which the arcs strictly increase and then strictly decrease in argument. Let $a_i$ be the arc with the largest argument. Since no arc with radius larger than $a_i$ has a line of sight to any arc with radius smaller than $a_i$, removing $v_i$ disconnects $G$. Moreover, the subgraph induced by the vertices $\{v_j\}$ with $j \geq i$ is a path, so it is in particular planar. The original graph $G$ is therefore planar if and only if the subgraph induced by the vertices $\{v_\ell\}$ with $\ell \leq i$ is. Call the resulting graph $G'$.

We will now show that the graph $G'$ is a minor of a graph $G''$ that has a planar embedding. Since taking minors preserves planarity, this will imply that $G'$ and hence $G$ are planar.

By assumption, no two endpoints of the remaining arcs lie on the same radial line. There therefore exists an $\ep > 0$ such that perturbing the positive endpoint of any arc by the angle $\ep$ does not change any visibilities. Let $m$ be an integer such that $2 \pi/m \ll \ep$, and adjust the argument of the right endpoint of every arc to the closest integer multiple of $2 \pi/m$. (In other words, take each arc and snap its endpoint to the vertex of a regular $m$-gon.) The choice of $m$ guarantees that this transformation does not alter any visibilities.

Call a semi-arc visibility representation with $n$ arcs \emph{polygonal} if the arcs strictly increase in argument and the endpoint of each arc lies on a vertex of a regular $n$-gon, and denote the corresponding graph by $H_n$. We claim that $G'$ is a minor of $H_m$. Indeed, the semi-arc visibility representation of $G'$ constructed above is a subset of the semi-arc visibility representation of $H_m$. As before, label the arcs of the polygonal representation of $H_m$ as $a_1, \dots, a_n$ with indices increasing with increasing radius, and label the vertices of $H_m$ as $v_1, \dots, v_n$ in a corresponding fashion. 

First, take a subgraph of $H_m$ by removing any arcs in the polygonal representation with larger radius than all the arcs in the representation of $G'$. Now, consider removing an arc $a_i$ from the representation of $H_m$. Since the arcs' arguments strictly increase, $v_i$ is not connected to $v_j$ for $j > i+1$. Moreover, once we remove $a_i$, any visibility between $a_{i+1}$ and an arc of smaller radius was either already a visibility of $a_{i+1}$ or corresponds to a visibility that used to exist with $a_i$. This implies that removing $a_i$ from the semi-arc visibility representation has the effect of contracting the edge between $v_i$ and $v_{i+1}$.

Removing from the polygonal representation all arcs which do not appear in the representation of $G'$ therefore yields a series of contractions of edges in $H_m$ yielding $G'$. So $G'$ is a minor of $H_m$, as claimed.

Finally, it suffices to show that $H_m$ is planar for all $m \geq 3$. Planar embeddings for these graphs appear in Figure~\ref{fig:planarhm}.
\begin{figure}
\centering
\begin{minipage}{.20\linewidth}
  \includegraphics[width=\linewidth]{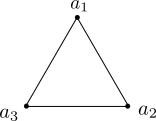}
    \end{minipage}
    \hfill
  \begin{minipage}{.35\linewidth}
  \includegraphics[width=\linewidth]{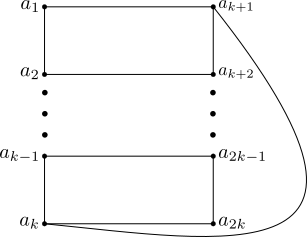}
    \end{minipage}
    \hfill
  \begin{minipage}{.3\linewidth}
  \includegraphics[width=\linewidth]{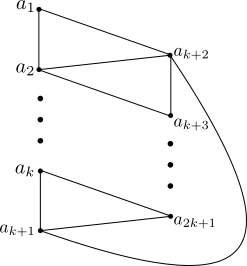}
  \end{minipage}
  \caption{Planar embeddings of $H_3$, $H_{2k}$, and $H_{2k+1}$.}
  \label{fig:planarhm}
\end{figure}
\end{proof}

Cobos et al.~\cite{cobos1996} gave a complete characterization of semi-bar visibility graphs. We extend their result to a complete classification of semi-arc visibility graphs.
\begin{definition}
A graph $G$ is \emph{outerhamiltonian} if it has a planar embedding in which there is path through all the vertices and the vertices on this path all lie on the outer face.
\end{definition}
\begin{definition}
Fix an outerhamiltonian graph $G$ and a corresponding hamiltonian path along the outer face. 
Label some subset (possibly empty) of the cutpoints of $G$ as $b_1,b_2,\dots,b_m$, with indices increasing in order along the path, and choose $1 \leq j \leq m$. A \emph{diagonal graph} of $G$ is a graph such that the following constraints hold.
\begin{itemize}
	\item{$b_1$ is connected to $b_j$ and $b_1$ is not connected to $b_i$ for $i<j$.}
	\item{If $b_i$ is connected to $b_r$ and $b_\ell$ then it is connected to $b_{k}$ for all $r<k<\ell$.}
	\item{For $1 \leq i < j$, let $b_k$ be the element with the highest index in sequence $b_1,b_2,\ldots,b_m$ such that $b_i$ is connected to $b_k$. Then $b_{i+1}$ is not connected to $b_1,\dots,b_{k-1}$ but is connected to either $b_k$ or $b_{k+1}$.}
\end{itemize}
\end{definition}
\begin{theorem}
A graph is a semi-arc visibility graph if and only if it is the union of an outerhamiltonian graph and an associated diagonal graph.
\end{theorem}
\begin{proof}
We first show that a semi-arc visibility graph has the claimed form.

Cobos et al.~established~\cite[Theorem 4]{cobos1996} that a graph is a semi-bar visibility graph if and only if it is outerhamiltonian. Given a semi-arc visibility representation, visibilities not through the center of the circle form a semi-bar visibility graph. It therefore suffices to show that the visibilities through the center of the circle form an associated diagonal graph.

If no arc has argument greater than $\pi$, then there are no visibilities through the center and the diagonal graph is empty. Otherwise, consider the arc with the greatest argument. (If there is more than one such arc, take the innermost one.) The corresponding vertex must be a cutpoint, since the arc blocks all possible visibilities between arcs with larger radius and arcs with smaller radius. Call this arc and its corresponding vertex $b_m$. Consider all arcs with smaller radius than $b_m$, and label the vertex corresponding to the largest such arc $b_{m-1}$. (If there are multiple candidates, take the innermost one.) Continuing in this way, construct the sequence of arcs $b_m,\dots,b_1$. By the same reasoning as given above, each such arc must correspond to a cutpoint.

Consider the arcs in this sequence with argument at least $\pi$ radius, and let $b_j$ be the arc in this subset with smallest index. Notice that $b_1$ and $b_j$ can view each other through the center of the circle, but $b_1$ cannot view $b_i$ with $i<j$ through the center as $b_i$ has argument less than $\pi$ radians. For $1 \leq i \leq m$, if $b_i$ views $b_r$ and $b_\ell$, then it views $b_{k}$ through the center for all $r<{k}<\ell$ as $b_k$ has radius and argument in between those of $b_r$ and $b_\ell$. Finally, for $1 \leq i < j$, let $b_k$ be the element with the highest index in sequence $b_1,\ldots,b_m$ such that $b_i$ is connected through the center to $b_k$. If $b_i$ and $b_k$ radially share the same endpoint, then $b_{i+1}$ views $b_{k+1}$ but nothing earlier in the sequence as these arcs are blocked by $b_i$. Otherwise $b_{i+1}$ views $b_k$ but nothing earlier in the sequence, as these arcs are blocked by $b_i$. The visibilities through the center therefore form a diagonal graph.

In the other direction, consider an outerhamiltonian graph and associated diagonal graph. The construction of Cobos et al.~\cite{cobos1996} shows how to represent the outerhamiltonian graph as a semi-bar visibility graph. Moreover, their construction shows that we can take the bars corresponding to the cutpoints all be of the same height, taller than all other bars.

If the diagonal graph is empty, then this semi-bar visibility representation immediately corresponds to a semi-arc visibility representation upon embedding the bars in the upper half of the circle. If the diagonal graph is not empty, lengthen the bars corresponding to the selected cutpoints $b_1,\dots,b_m$ so that their lengths increase from left to right; if a cutpoint not belonging to the sequence $b_1,\dots,b_m$ lies between $b_i$ and $b_{i+1}$, give it the same length as $b_i$. This representation has the same visibilities as the original graph. Transform this representation into a semi-arc visibility representation by stretching the bars into arcs while maintaining their relative lengths, so that $b_j$ has argument $\pi$. Then it is easy to check that visibilities through the center form the given diagonal graph. The claim follows.
\end{proof}

\section{Improved edge bounds for arc and semi-arc $k$-visibility graphs}
\label{sec:edge}
Babbitt et al.~\cite{babbitt2013k} established upper bounds on the total number of edges for arc and semi-arc $k$-visibility graphs. For arc $k$-visibility graphs, they proved that a graph with $n$ vertices can have at most $(k+1)(3n-k-2)$ edges. We improve this bound to $(k+1)(3n-\frac{3k+6}{2})$. For semi-arc $k$ visibility graphs, they proved that a graph with $n$ vertices can have at most $(k+1)\left(2n-\frac{k+2}{2}\right)$ but conjectured that the correct bound was smaller. We prove that in fact their original bound is tight. These results are summarized in Table~\ref{table:edge_bounds}.

\begin{table}
\begin{center}
	\begin{tabular}{@{}lcc@{}}\toprule
	&arc&semi-arc\\ \midrule
	Babbitt et al. &$\le(k+1)(3n-k-2)$ &$\le(k+1)\left(2n-\frac{k+2}{2}\right)$\\
	This Work & $\le(k+1)\left(3n-\frac{3k+6}{2}\right)$&\quad$(k+1)\left(2n-\frac{k+2}{2}\right)$\\
	\bottomrule
	\end{tabular}
	\caption{Maximum number of edges in arc and semi-arc $k$-visibility graphs on $n$ vertices.}
	\label{table:edge_bounds}
\end{center}
\end{table}

Since we seek to establish upper bounds on the number of edges, we assume in this section that each arc has a different radius and moreover that no two endpoints of any two arcs lie on the same radial segment. We can accomplish this without decreasing the number of edges by slightly perturbing arcs and their endpoints.

We begin by establishing some definitions. Fix an arc $k$-visibility representation of a graph and consider two arcs joined by some line of sight. Consider the set of all valid lines of sight between the two arcs. Each such line can be uniquely associated with an argument $\theta \in (- \pi, \pi]$ denoting the angle that line makes with the positive $x$-axis. 

Two lines of sight are \emph{contiguous} if we can rotate one into the other such that all intervening segments are also valid lines of sight, and we define a \emph{region of visibility} to be the closure of a maximal set of contiguous lines of sight. For each region of visibility, we call the radial segment in it with the smallest argument the \emph{limiting line} of the region.

Following Babbitt et al.~\cite{babbitt2013k}, we associate edges in an arc $k$-visibility graph $G$ with arcs according to the limiting lines of their regions of visibility. (Our definitions differ from theirs in that we allow an edge to be associated with multiple arcs.)

Fix an arc $k$-visibility representation of $G$. Suppose that two arcs $a_u$ and $a_v$ in the representation are connected by a line of sight, so that the corresponding vertices $u$ and $v$ are connected by an edge in $G$. We consider each region of visibility between $a_u$ and $a_v$ in turn. Given a region, if the limiting line contains an endpoint of $a_u$ (respectively $a_v$), then we call the edge between $u$ and $v$ in $G$ a \emph{negative edge} of $a_u$ (respectively $a_v$). Otherwise, the limiting line must contain the endpoint of another arc, say $a_w$. In this case, we call the edge between $u$ and $v$ a \emph{positive edge} of $a_w$. In this way, we associate each edge to (possibly many) arcs in the visibility representation.

In Figure~\ref{fig:posneg}, we give an example showing several limiting lines corresponding to regions of visibility in an arc $k$-visibility graph. In each case, the edge is assigned to the arc whose endpoint is contained in the limiting line.

\begin{figure}
\begin{center}
\includegraphics[width=.17\textwidth]{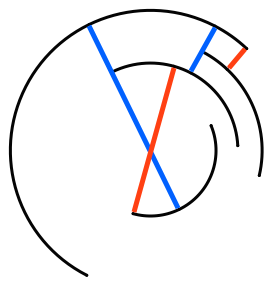}
\caption{Positive (red) and negative (blue) edges}
\label{fig:posneg}
\end{center}
\end{figure}

The following lemma establishes a link between the number of regions of visibility between two arcs and the number of arcs that the corresponding edge is assigned to.

\begin{lemma}\label{lem:regions}
Fix a pair of arcs $a_u$ and $a_v$ in an arc $k$-visibility representation corresponding to a pair of connected vertices $u$ and $v$ in $G$. Suppose that the arcs have $m$ distinct regions of visibility. Then the edge between $u$ and $v$ is assigned to at least $m$ arcs.
\end{lemma}

\begin{proof}
It suffices to show that the limiting line of each region of visibility corresponds to a distinct arc. To see this, note that each assignment corresponds to the endpoint of some arc, and that endpoint cannot lie in any other region of visibility between $a_u$ and $a_v$.

Hence different regions of visibility correspond to distinct assignments of edges to arcs. The claim follows.
\end{proof}
\subsection{An improved bound for arc $k$-visibility graphs}
Using Lemma~\ref{lem:regions} it is possible to improve the bound on the maximum number of edges given in Babbitt et al.~\cite{babbitt2013k} for arc $k$-visibility graphs.
\begin{theorem}\label{thm:arcbound}
The maximum number of edges in an arc $k$-visibility graph with $n$ vertices is at most $\dbinom{n}{2}$ for $n\le{4k+4}$ and $(k+1)(3n-\frac{3k+6}{2})$ for $n>4k+4$.
\end{theorem}
\begin{proof}
When $n \leq 4k + 4$ the bound is trivial, so we can assume that $n > 4k+4$.

Label the outermost $k+1$ arcs $a_{k+1},\ldots,a_1$ with indices decreasing with decreasing radius. We first recall how to obtain the bound given in Babbitt et al.~\cite{babbitt2013k}. Note that there are at most $2k+2$ negative edges and $k+1$ positive edges associated to each arc. Moreover, for the outermost arcs a stronger bound holds: there are at most $0,1,\ldots,k$ positive edges and $k+1,k+2,\ldots,2k+1$ negative edges for $a_{k+1},\ldots,a_1$ respectively. The total number of edges is therefore at most $(3k+3)n - 2\sum_{i=1}^{k+1} i = (k+1)(3n - k - 2)$.
\begin{figure}
\centering
\begin{minipage}{.13\linewidth}
\includegraphics[width=\linewidth]{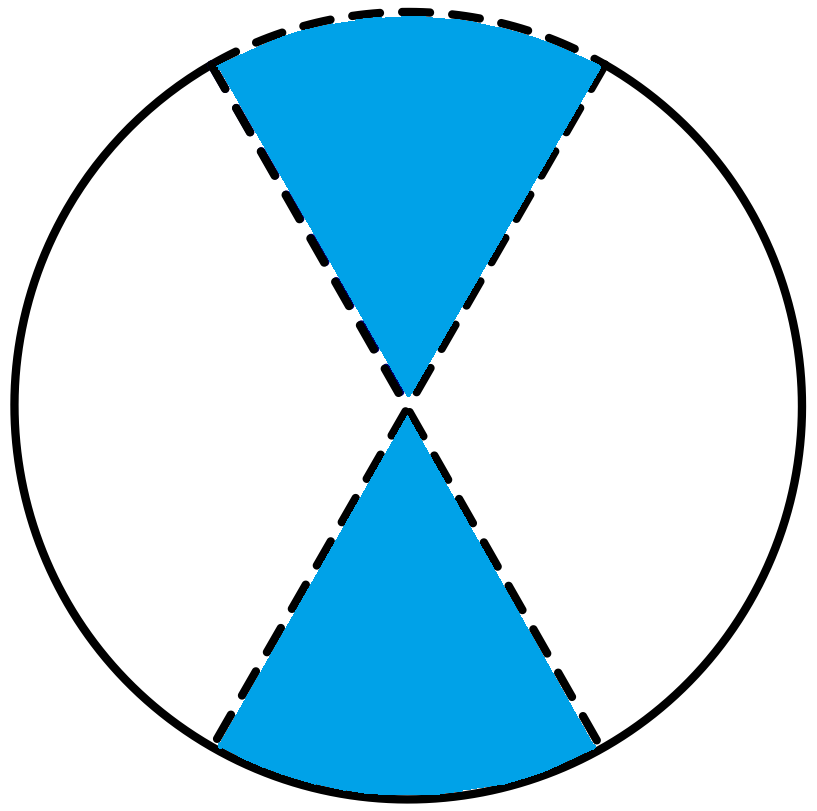}
\end{minipage}
\hspace{.05\linewidth}
\begin{minipage}{.13\linewidth}
\includegraphics[width=\linewidth]{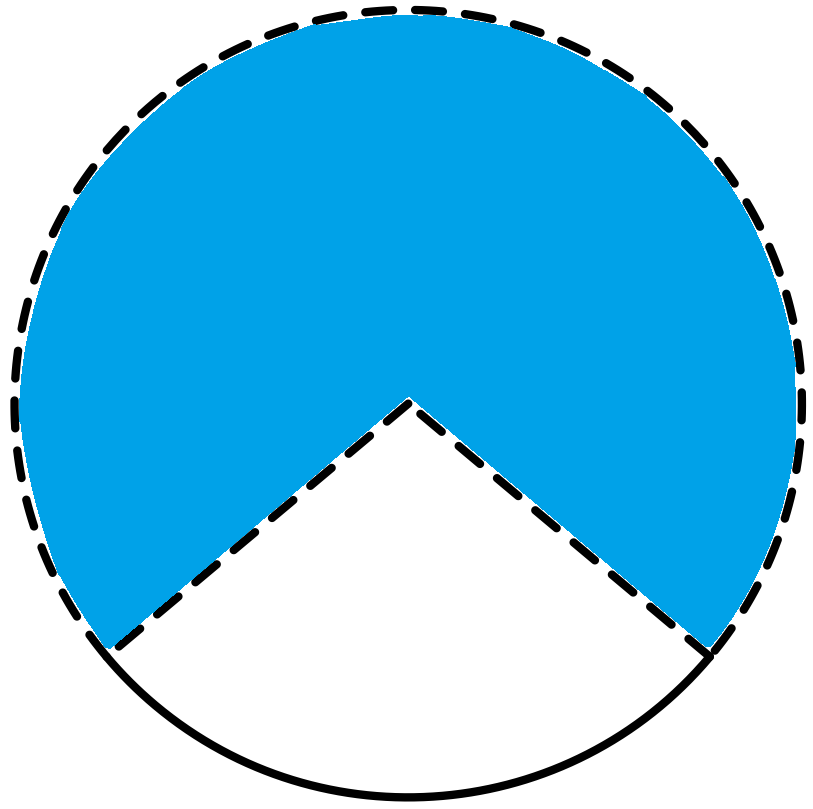}
\end{minipage}
\caption{Cone of Visibility}
\label{fig:cone}
\end{figure}
Note that the above bound still holds if we replace the $k+1$ outermost arcs by circles. We will replace the outermost by circles inductively from the outside in, and if this process adds to the total number of edges then we will have shown that the original bound was not tight.

Consider the arc $a_\ell$, and suppose that we have already closed the $k+1-\ell$ outermost arcs into circles. (If $\ell = k+1$, then we simply consider the original represenation.) Say an arc $a$ is in the \emph{cone of visibility} of $a_\ell$ if there is a radial line of sight from $a$ to the exterior of the circle on which $a_\ell$ lies that does not intersect $a_\ell$. (In Figure~\ref{fig:cone}, any arc in the blue shaded region is in the cone of visibility of the outer arc.)

Suppose that there are $m$ arcs contained in the cone of visibility. Consider arcs in the order shown in Figure~\ref{fig:closer}, where the two different cases correspond to whether $a_\ell$ is an obtuse or acute arc, and take the $\ell$ arcs encountered last in moving from the tail to the head of the arrow. (If $m < \ell$, take all the arcs.) Call this set of arcs $S_\ell$. It is easy to see that all arcs in $S_\ell$ have a valid line of sight to the outermost circle passing through at most $k-1$ arcs or circles.

Consider any arc $a$ in $S_\ell$. Suppose there is a radial line passing through the negative endpoint of $a$ that does not intersect $a_\ell$. Then when we close the arc $a_\ell$ into a circle, we add a negative edge to $a$. Since $a$ had a line of sight to the outermost circle passing through at most $k-1$ arcs before this operation, $a$ still has a line of sight to the outermost circle. Therefore this operation has added a negative visibility to $a$ without removing any visibilities of $a$, so the original graph was missing this visibility.

In a similar way, if there is a radial line passing through the positive endpoint of $a$ but not intersecting $a_\ell$, then the original graph was missing a positive visibility associated with $a$. 

Finally, if radial lines from both endpoints of $a$ intersect $a_\ell$, then $a$ and $a_\ell$ have at least two distinct regions of visibility. Lemma~\ref{lem:regions} then implies that the edge between $a$ and $a_\ell$ is assigned to at least $2$ arcs.

The proceeding considerations show that each arc in $S_\ell$ is associated with at least $1$ extra visibility, either one that is missing in the original graph or is counted twice in the original bound. Moreover, if $|S_\ell| < \ell$, then $a_\ell$ itself is missing at least $\ell - |S_\ell|$ positive visibilities. In any case, we obtain that $a_\ell$ is associated with an over-counting of at least $\ell$ visibilities. Repeating this process for all $\ell$ with $1 \leq \ell \leq k+1$ yields a total over-count of $\sum_{\ell = 1}^{k+1} \ell = (k+1)\left(\frac{k+2}{2}\right)$. We therefore obtain that the maximum number of edges is $$(k+1)\left(3n - k - 2 - \frac{k+2}{2}\right) = (k+1)\left(3n-\frac{3k+6}{2}\right)\,,$$ as desired.

\begin{figure}
\centering
\begin{minipage}{.13\linewidth}
   \includegraphics[width=\linewidth]{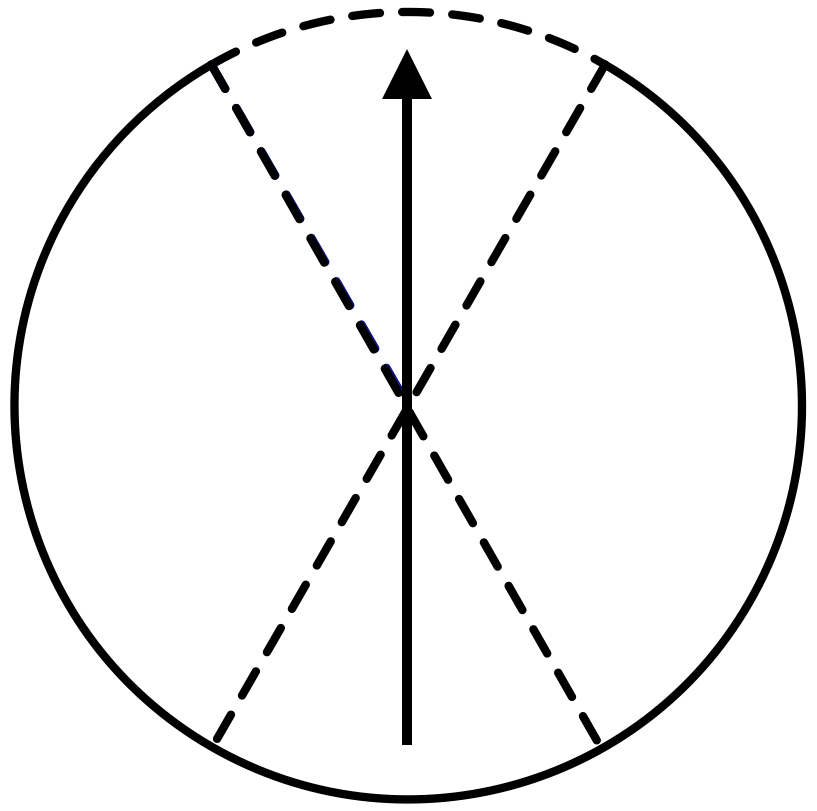}
\end{minipage}
\hspace{.05\linewidth}
\begin{minipage}{.13\linewidth}
  \includegraphics[width=\linewidth]{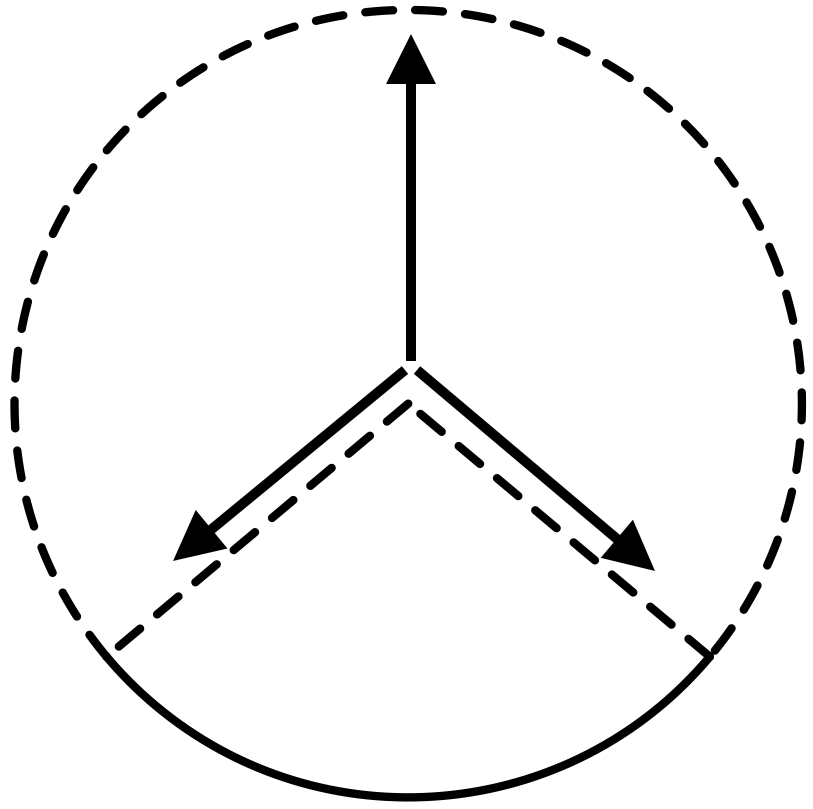}
\end{minipage}
\caption{Definition of $S_\ell$ for obtuse and acute arcs}
\label{fig:closer}
\end{figure}

\end{proof}
\begin{corollary}
\label{cor:tightarcedge}
The maximum number of edges in an arc visibility graph with $n$ vertices is $\dbinom{n}{2}$ for $n\le5$ and $3n-3$ for $n\ge6$. This bound can be achieved as shown in Figure~\ref{fig:arcmax} with the dots indicating any necessary additional arcs. (If $n<5$ take the innermost $n$ arcs in Figure~\ref{fig:arcmax}.)
\end{corollary}
\begin{figure}
\begin{center}
\includegraphics[width=.27\textwidth]{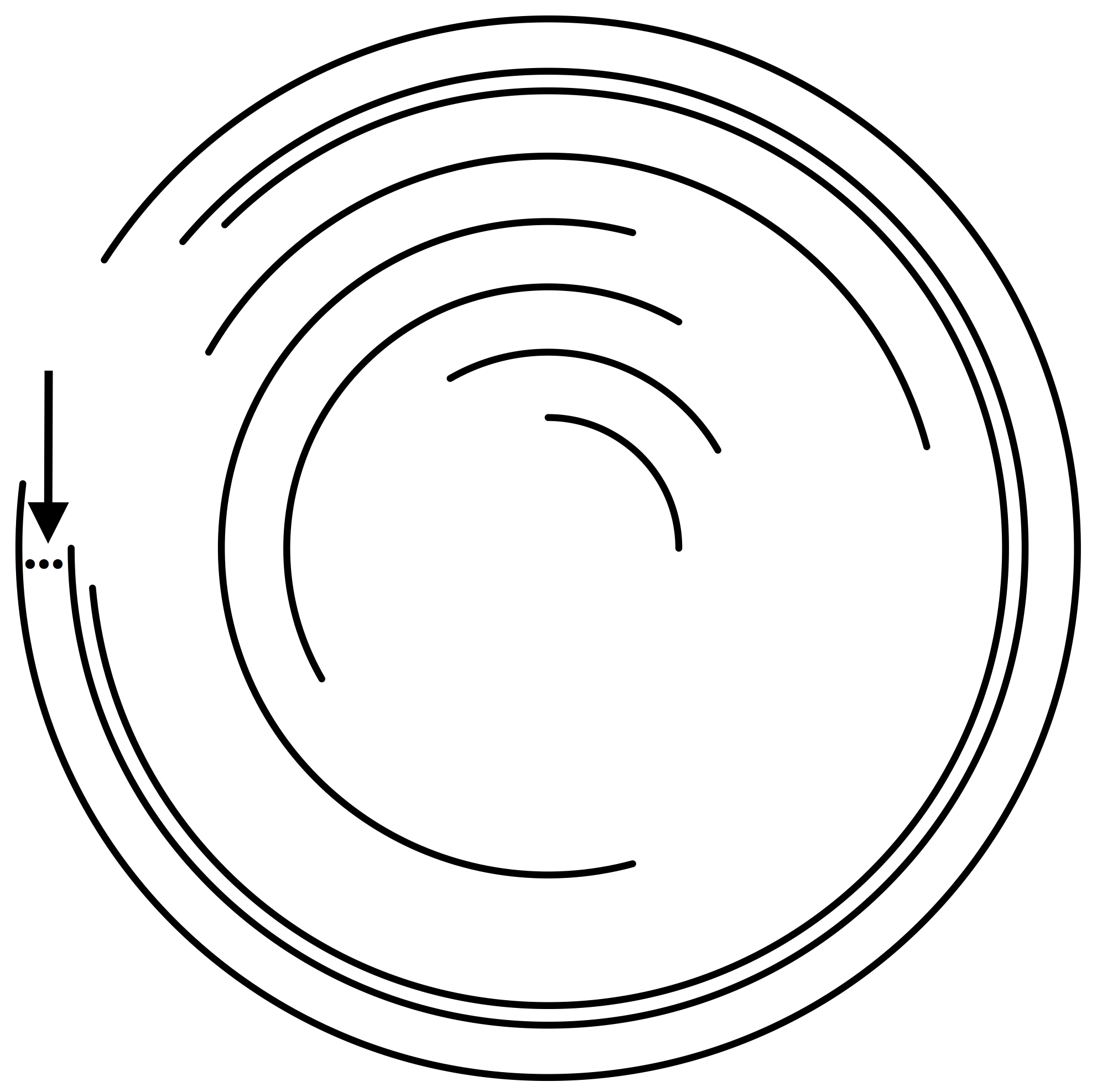}
\caption{Arc visibility representation with maximum number of edges}
\label{fig:arcmax}
\end{center}
\end{figure}
\subsection{A tight construction for semi-arc $k$-visibility graphs}
As noted above, the edge bound previously given in Babbitt et el.~\cite{babbitt2013k} for semi-arc $k$-visibility graph is actually optimal. By establishing optimality, we disprove their Conjecture 20, which posited a smaller upper bound.
\begin{theorem}\label{thm:semiarck_bound}
The maximum number of edges in a semi-arc $k$-visibility graph with $n$ vertices is $(k+1)\left(2n-\frac{k+2}{2}\right)$ for $n\ge5k+5$ and this bound is optimal.
\end{theorem}
\begin{proof}
\par{The maximum number of edges in a semi-arc $k$-visibility graph with $n$ vertices is at most $(k+1)\left(2n-\frac{k+2}{2}\right)$ for $n\ge3k+3$. This is Theorem 13 in Babbitt et al.~\cite{babbitt2013k}.}
\begin{figure}
\begin{center}
\includegraphics[width=.30\textwidth]{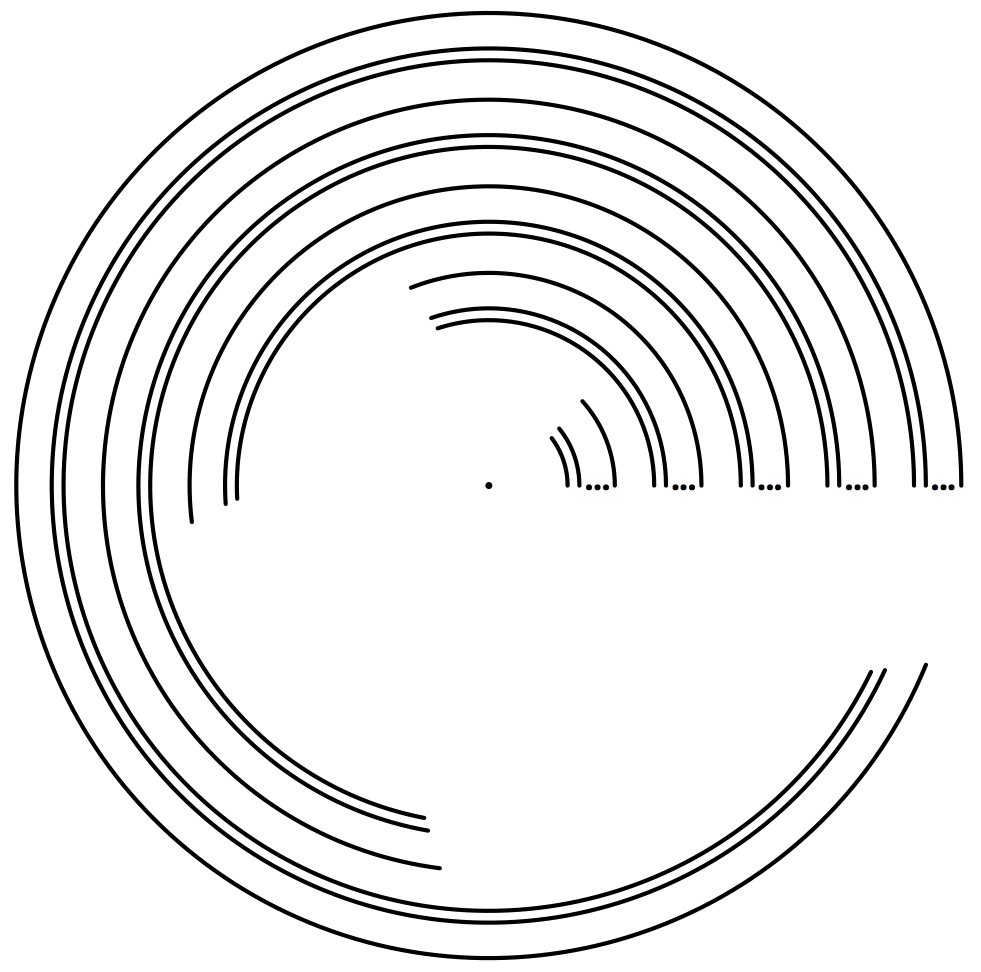}
\caption{Semi-arc $k$-visibility representation with $5k + 5$ arcs and the maximum number of edges (Each set of arcs has $k+1$ arcs.)}
\label{fig:semimax}
\end{center}
\end{figure}
We claim that the semi-arc $k$-visibility representation in Figure~\ref{fig:semimax} proves the claim for $n=5k+5$. Let the arcs be marked $a_1,\ldots,a_{5k+5}$ with indices increasing with increasing radius in Figure 12. (The arcs $a_1,\ldots,a_{5k+5}$ have arguments of $\frac{\pi}{5},\frac{\pi}{5}+\epsilon,...,\frac{\pi}{5}+k\epsilon,\frac{3\pi}{5},...,\frac{3\pi}{5}+k\epsilon,\pi,...,\pi+k\epsilon,\frac{7\pi}{5},...\frac{7\pi}{5}+k\epsilon,\frac{9\pi}{5},...,\frac{9\pi}{5}+k\epsilon$ radians respectively for $\epsilon$ sufficiently small.) This semi-arc $k$-visibility representation gives a total of $(k+1)\left(2(5k+5)-\frac{k+2}{2}\right)$ edges as there are $5(k+1)^2$ edges corresponding to visibilities through the center and $(k+1)\left(\frac{9k+8}{2}\right)$ edges corresponding to visibilities not through the center. This establishes the desired edge count for $n=5k+5$. For $n>5k+5$, add an additional $n-5k-5$ arcs between $a_{3k+3}$ and $a_{3k+4}$ that have an argument less than that of $a_1$. Notice that each new arc adds $2k+2$ edges so the bound is optimal for $n>5k+5$.
\end{proof}

Babbitt et al.~\cite{babbitt2013k} also conjectured that the complete graph $K_{3k+4}$ is not a semi-arc $k$-visibility graph. Using a construction similar to the one given above, we disprove this conjecture as well.
\begin{theorem}
\label{thm:3k4}
$K_{3k+4}$ is a semi-arc $k$-visibility graph. 
\end{theorem}
\begin{proof}
The construction is given in Figure~\ref{fig:semiarc_complete}. (As radius increases, the arcs have arguments of $\frac{\pi}{3},\frac{\pi}{3}-\epsilon,...,\frac{\pi}{3}-k\epsilon,\frac{2\pi}{3},...,\frac{2\pi}{3}-k\epsilon,\pi,...,\pi+k\epsilon,\frac{5\pi}{3}$ radians for $\epsilon$ sufficiently small.)
\end{proof}
\begin{figure}
\begin{center}
\includegraphics[width=.27\textwidth]{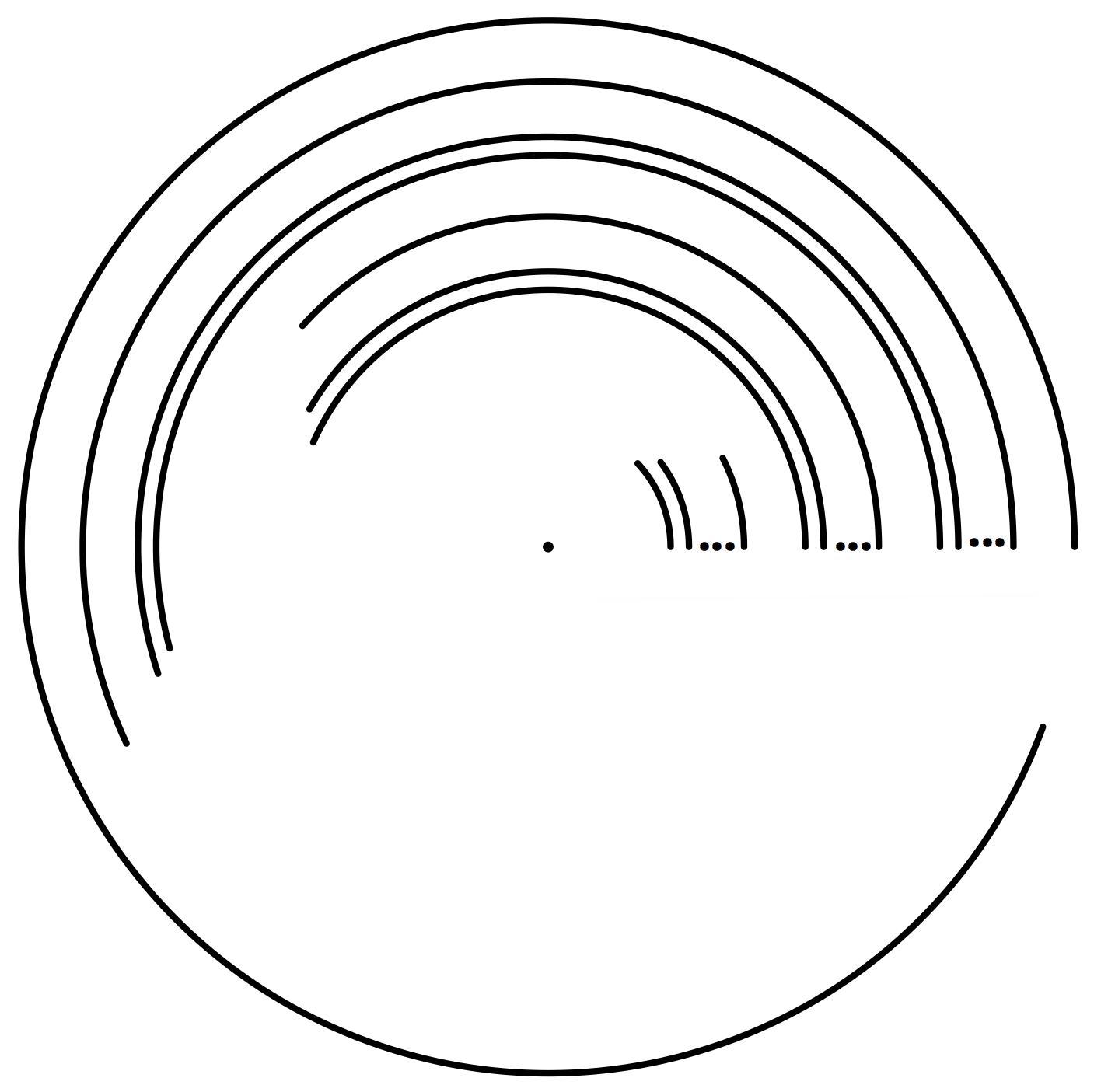}
\caption{Semi-arc $k$-visibility representation of $K_{3k+4}$ (Each set of arcs has $k+1$ arcs.)}
\label{fig:semiarc_complete}
\end{center}
\end{figure}

Combining Theorem~13 in Babbitt et al.~\cite{babbitt2013k} with Theorems~\ref{thm:semiarck_bound} and~\ref{thm:3k4}, we have a construction of a semi-arc $k$-visibility graph on $n$ vertices with the maximum number of edges for $n \leq 3k+4$ and $n \geq 5k+5$. This leaves open the question of finding a maximum construction (or proving an improved bound) for $3k + 4 < n < 5k+5$. When $k=0$, there is no gap; when $k=1$ the only open cases are $n = 8$ and $n = 9$.

\section{Thickness bounds}
\label{sec:thick}
In this section, we prove new bounds on the thickness of arc and semi-arc $k$-visibility graphs.

\begin{definition}
The \emph{thickness} of a graph $G$, denoted $\theta(G)$, is the smallest number of planar graphs into which the edges of $G$ can be partitioned.
\end{definition}

Bounding the thickness of bar $k$-visibility graphs has been a main subject of interest ever since their introduction by Dean et al.~\cite{dean2007bar}. This quantity is especially relevant to VLSI design, where graphs of low thickness correspond to circuit designs that are electrically practical~\cite{mutzel1998}.

Computing the thickness of a graph is \textsc{np}-hard in general~\cite{mansfield1983}, and exact thickness results are still open for all but a few classes of visibility graphs. Recently, Chang et al.~\cite{chang} proposed using a simpler quantity, arboricity, to obtain easier bounds on thickness purely in terms of extant edge bounds. The results of this section use this strategy and the results of Sections~\ref{sec:class} and~\ref{sec:edge} to prove new thickness bounds. (See Table~\ref{table:thick}.)

We first review some basic facts about arboricity.

\begin{definition}
The arboricity of a graph $G$, denoted $\arb(G)$, is the smallest number of forests into which the edges of a graph can be partitioned.
\end{definition}

Since forests are planar graphs, the thickness of $G$ is at most its arboricity. Moreover, it is easy to see that if the graph $H$ is planar, then $\arb(H) \leq 3$, so the arboricity of a graph is at most three times its thickness.

Unlike thickness, arboricity has a good characterization, given originally by Nash-Williams~\cite{nash1961edge}.

\begin{theorem}[Nash-Williams Theorem]
For any graph $G$,
$$\arb(G)=\max\limits_{H\subseteq{G}}\left\lceil{\frac{E_H}{N_H-1}}\right\rceil,$$ where $N_H$ and $E_H$ are the number of vertices and edges respectively in the subgraph $H$.
\end{theorem}

Though the statement of the Nash-Williams Theorem appears to require checking an exponential number of subgraphs, calculating the arboricity of a graph is a special case of finding the minimal partition of a matroid into independent sets, which can be done in polynomial time~\cite{edmonds1965}.
\begin{table}
\begin{center}
	\begin{tabular}{@{}lcccc@{}}\toprule
	&bar&semi-bar&arc&semi-arc \\ \midrule
	Dean et al.&$\le3k(6k+1)$&\\
	Chang et al.&$\le3k+3$&$\le2k$\\
	Babbitt et al. &&$\le2k$&\\
	Our Work &&&$\le3k+3$&$\le2k+1$   \\
	\bottomrule
	\end{tabular}
	\caption{Maximum thickness of bar, semi-bar, arc, and semi-arc $k$-visibility graphs.}\label{table:thick}
\end{center}
\end{table}

We obtain the following theorem.
\begin{theorem}
\label{thm:arckthickness}
The thickness of an arc $k$-visibility graph is at most $3k+3$.
\end{theorem}
\begin{proof}
Let $G$ be an arc $k$-visibility graph on $n$ vertices, and let $H \subseteq G$ have $\ell$ vertices. Removing all arcs from the visibility representation of $G$ except those corresponding to the vertices of $H$ yields an arc $k$-visibility graph $G'$ on $\ell$ vertices such that $H \subseteq G'$. So we can assume that $H$ is a subgraph of a arc $k$-visibility graph with the same number of vertices.

By Theorem~\ref{thm:arcbound}, $H$ has at most $(k + 1)\left(3 N_H - \frac{3k+6}{2}\right)$ edges if $\ell > 4k+4$ and $\binom{N_H}{2}$ otherwise. In the former case, 
\begin{equation*}
E_H \leq (k + 1)\left(3 N_H - \frac{3k+6}{2}\right) = (3k+3)(N_H - 1 - \frac k 2) \leq (3k+3)(N_H - 1).
\end{equation*}
In the latter case, $E_H = \frac{N_H}{2}(N_H - 1) \leq (2k+2)(N_H - 1)$. 

The Nash-Williams Theorem then yields
\begin{equation*}
\theta(G) \leq \arb(G) = \max\limits_{H\subseteq{G}}\left\lceil{\frac{E_H}{N_H-1}}\right\rceil \leq 3k+3\,
\end{equation*}
as desired.
\end{proof}

\begin{corollary}
\label{cor:arcthick}
The thickness of an arc visibility graph is at most $3$.
\end{corollary}

Note that Corollary~\ref{cor:arcthick} is stronger that what could have been obtained by applying the Nash-Williams Theorem to the bound of $3n-2$ for arc visibility graphs proved by Babbitt et al.~\cite{babbitt2013k}, which yields a maximum thickness of $4$.

Applying the above strategy to semi-arc $k$-visibility graphs using the edge bound in Babbitt et al.~\cite{babbitt2013k} (which we showed to be tight in Section~\ref{sec:edge}) shows that the thickness of these graphs is at most $2k+2$ for $k \geq 2$. Using the classification of semi-arc visibility graphs given in Section~\ref{sec:class}, we can obtain a stronger statement.

\begin{theorem}
The thickness of a semi-arc $k$-visibility graph is at most $2k+1$.
\end{theorem}
\begin{proof}
Fix a semi-arc $k$-visibility graph $G$ and an associated representation. As in Section~\ref{sec:edge}, we can assume that all arcs have radially distinct endpoints and distinct radii since we can achieve this by small perturbations without decreasing the thickness of the graph.

Given such a representation, call the semi-arc ($0$-)visibility graph associated with the collection of arcs $SA_0$. Note that $SA_0 \subseteq G$. Moreover, by Theorem~\ref{thm:arcplanar} $SA_0$ is planar. 

Remove the edges in $SA_0$ from $G$ and call the remaining graph $G'$. For every pair of vertices connected by an edge in $G'$, the line of sight with largest argument between the corresponding arcs contains one of their endpoints. Direct all edges in $G'$ from the arc whose endpoint is contained in the corresponding line of sight to the one whose endpoint is not contained in the line of sight. Each vertex in this graph has outdegree at most $2k$, so this graph can be partitioned into $2k$ disjoint graphs in which each vertex has outdegree at most $1$. It is easy to see that all such graphs are planar. So the thickness of $G'$ is at most $2k$. We obtain $\theta(G) \leq \theta(G') + \theta(SA_0) \leq 2k + 1$, as desired.
\end{proof}

\section{Comparison of families of arc $k$-visibility graphs and bar $k$-visibility graphs}
\label{sec:inclusion}
In this section, we consider the relationship between bar and arc visibility graphs. We first show several structural properties of the family of bar visibility graphs and then use these results to show that arc visibility graphs are not bar visibility graphs in general. On the other hand, we note that bar $j$-visibility graphs are a subset of arc $j$-visibility graphs, since any bar visibility representation can easily be converted into a corresponding arc visibility representation. (This observation appears in Babbitt et al.~\cite{babbitt2013k}.)

We begin by analyzing the families of bar $i$-visibility and $j$-visibility graphs for $i \neq j$. A result of Hartke et al.~\cite{hartke2007further} establishes that these families are incomparable under set inclusion when $j = i+1$. We use a similar construction to generalize their argument to all $i \neq j$.

We require one definition. \emph{Interval graphs} are generalizations of bar $k$-visibility graphs where lines of sight are allowed to pass through an unlimited number of intervening bars. Following our practice above, we call the collection of bars corresponding to an interval graph an \emph{interval representation}. Interval graphs are well studied and have been completely characterized. In particular, it is known that all interval graphs are chordal (that is, contain no induced cycle of length more than three)~\cite{gilmore1964}.

The following theorem provides a precise connection between $k$-visibility graphs and interval graphs.

\begin{theorem}
\label{thm:intbark}
Let $G$ be a $K_\ell$-free graph for $\ell \leq k + 2$. Then $G$ is an interval graph if and only if it is a bar $k$-visibility graph.
\end{theorem}
\begin{proof}
Suppose that $G$ is an interval graph, and fix an interval representation. Since $G$ is $K_\ell$ free, no vertical line intersects $\ell$ bars. Any line of sight in the interval representation therefore passes through at most $\ell - 2 \leq k $ intervening bars. So this set of bars is also a representation of $G$ as a bar $k$-visibility graph.

Conversely, assume $G$ is a bar $k$-visibility graph and fix a representation. If there existed a vertical line intersecting $\ell$ bars, then the corresponding vertices would form a copy of $K_\ell$, since every pair of bars would be separated by at most $\ell - 2 \leq k$ intervening bars. Therefore any pair of bars intersected by a vertical line are separated by at most $k$ bars, hence the corresponding vertices are connected in $G$. So this set of bars is also an interval representation for $G$ as an interval graph.
\end{proof}

Evans et al. observed~\cite[Lemma 1]{evans2013bar} that triangle-free bar $1$-visibility graphs are forests. Theorem~\ref{thm:intbark} implies the following stronger statement.

\begin{corollary}
If a bar $k$-visibility graph ($k\ge1$) is a triangle-free graph then it is a disjoint union of caterpillars. (A \emph{caterpillar} is a tree in which all vertices are within one edge of a central path.)
\end{corollary}
\begin{proof}
Theorem~\ref{thm:intbark} implies that a triangle-free bar $k$-visibility graph with $k \geq 1$ is an interval graph, and any triangle-free interval graph is a union of caterpillars~\cite{eckhoff1993}.
\end{proof}

We can now state the main theorem.

\begin{theorem}
\label{thm:bibj}
Let $B_{k}$ be the family of bar $k$-visibility graphs for $k \geq 0$. Then $B_{i}\not\subseteq{B_{j}}$ and $B_{j}\not\subseteq{B_{i}}$ for $i \neq j$.
\end{theorem}
\begin{proof}
Without loss of generality let $i>j$. Consider the graph $K_{j} \times {C_4}$, the tensor product of the complete graph on $j$ vertices with the cycle graph on $4$ vertices. Note that $K_{j} \times {C_4}$ is $K_{j+3}$ free.

Since $K_{j} \times {C_4}$ contains an induced four-cycle, it is not chordal and in particular is not an interval graph. Since $j+3 \leq i +2$, Theorem~\ref{thm:intbark} implies $K_{j} \times {C_4}$ is not a bar $i$-visibility graph. However, Figure~\ref{fig:tensor} shows that it is a bar $j$-visibility graph. Hence $B_i \not\subseteq B_j$.

\begin{figure}
\begin{center}
\includegraphics[width=.35\textwidth]{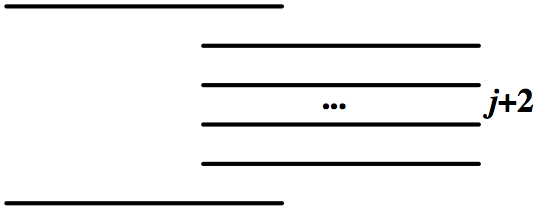}
\caption{Bar $j$-visibility representation of $K_{j}\times{C_4}$}\label{fig:tensor}
\end{center}
\end{figure}

On the other hand, Hartke et al.~\cite{hartke2007further} show that $K_{4j+4}\notin{B_i}$ and $K_{4j+4}\in{B_j}$. Therefore $B_j\not\subseteq{B_i}$.
\end{proof}

Let $A_k$ be the family of arc $k$-visibility graphs. Since $B_k \subset A_k$ for all $k$, Theorem~\ref{thm:bibj} implies that $A_j \not \subseteq B_i$ for all $i \neq j$. A more careful analysis shows that in fact this claim holds for all $i$ and $j$.
\begin{theorem}
\label{thm:ajbi}
Let $B_{k}$ be the family of bar $k$-visibility graphs and $A_{k}$ be the family of arc $k$-visibility graphs. Then $A_{j}\not\subseteq{B_{i}}$ for all $i, j \geq 0$.
\end{theorem}
\begin{figure}
\begin{minipage}[t]{.45\linewidth}
\centering
\includegraphics[width=1.5in]{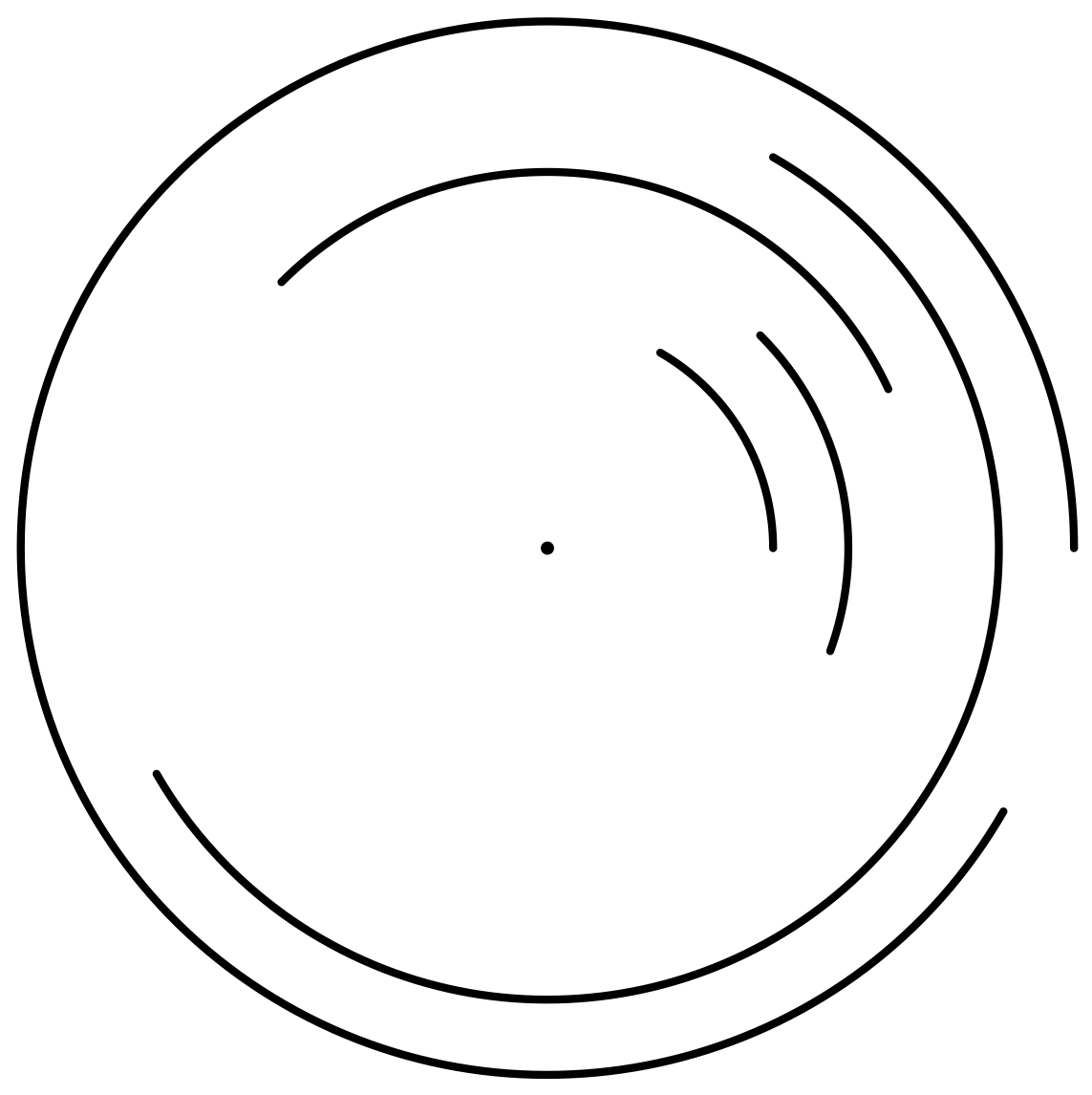}
\caption{Arc $j$-visibility representation of $K_5$}
\label{fig:k5}
\end{minipage}
\hfill
\begin{minipage}[t]{.45\linewidth}
\centering
\includegraphics[width=1.5in]{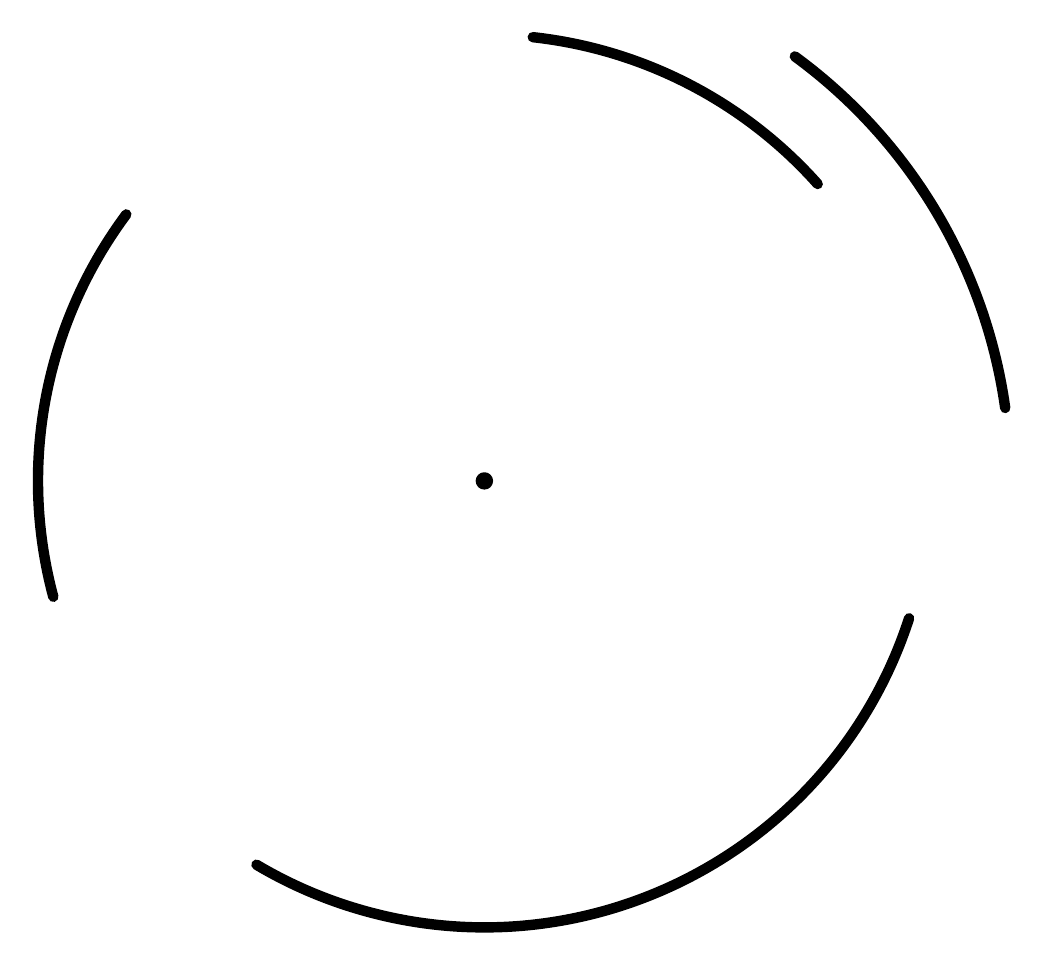}
\caption{Arc $j$-visibility representation of $C_4$}
\label{fig:c5}
\end{minipage}
\end{figure}
\begin{proof}
Fix a nonnegative $j$. We first show that $A_j \not \subseteq B_0$. Figure~\ref{fig:k5} shows that $K_5$ is an arc $j$-visibility graph for any $j \geq 0$. Since all bar visibility graphs are planar, this implies $A_j \not \subseteq B_0$.

Suppose now that $i \geq 1$. The cycle graph $C_4$ is triangle free, so Theorem~\ref{thm:intbark} implies that it is a bar $i$-visibility graph if and only if it is an interval graph. Since $C_4$ is not chordal, we conclude that $C_4$ is not a bar $i$-visibility graph for all $i \geq 1$. But Figure~\ref{fig:c5} shows it is an arc $j$-visibility graph. Therefore $A_j \not \subseteq B_i$, as desired.
\end{proof}

We note that the analogous questions for semi-bar $k$-visibility graphs and semi-arc $k$-visibility graphs are far simpler.
The family of semi-bar (semi-arc) $i$-visibility graphs is never contained in the family of semi-bar (semi-arc) $j$-visibility graphs for $i \neq j$ because a semi-bar or semi-arc $k$-visibility graph on $n$ vertices has at least $(k+1)n-O(1)$ edges and at most $2(k+1)n+O(1)$ edges.
\section{Random semi-bar and semi-arc visibility graphs}
\label{sec:random}
In this section, we consider random versions of semi-bar and semi-arc $k$-visibility graphs. We begin with semi-bar graphs, since these are a subfamily of semi-arc graphs.

To model random semi-bar $k$-visibility graphs, we first note that semi-bar $k$-visibility representations are in one-to-one correspondence with elements of $S_n$, the symmetric group on $n$ letters, since a representation can be uniquely defined by giving the relative lengths of the bars as they appear from top to bottom. This motivates the following definition.

\begin{definition}
A \emph{random semi-bar $k$-visibility graph on $n$ vertices} is the random graph corresponding to the semi-bar $k$-visibility representation generated by letting the right endpoints of $n$ semi-bars be drawn i.i.d.\ uniformly from $(0, 1)$. Call the resulting distribution on graphs $\mathcal G_n^k$.
\end{definition}

The same distribution can also be obtained by choosing an element of $S_n$ uniformly at random and constructing the semi-bar $k$-visibility representation corresponding to that permutation.

\begin{theorem}
\label{thm:expectedsemibar}
Let $G \sim \mathcal G_n^k$ and let $E$ be its number of edges. Then $\mathbb E[E]=\binom{n}{2}$ for $n\le{k+2}$ and $$\mathbb E[E]=\frac{1}{2}(k+1)\left(4n-3k-6-2(k+2)\sum_{l=k+3}^n \frac{1}{l}\right) = (k+1)(2n - o(n))$$ for $n\ge{k+3}$. Moreover, for any $t \geq 0$, 
\begin{equation*}
\mathbb P(|E - \mathbb E[E]| > (k+1)t) \leq 2 \exp\left(-\frac{2 t^2}{n}\right)\,.
\end{equation*}
\end{theorem}
\begin{proof}
If $n \leq k+2$, then $G$ is the complete graph and the claims are trivial. So suppose that $n \geq k+3$. 

Since drawing $G \sim \mathcal G_n^k$ is equivalent to drawing a permutation uniformly at random, we can generate $G$ by generating a permutation one element at a time. In each of $n$ rounds, we add a bar, shorter than all those added thus far, to a semi-bar visibility representation in a random position. Bars added in this way do not affect the visibilities already present in the graph, so it suffices to consider those added by the addition of the new bar.

In general, the addition of a new bar adds $2k+2$ edges, except when the new bar has fewer than $k+1$ bars to its right or left. If $m$ bars have already been added, then there are $m+1$ possible positions for the new bar, each equally likely. If $m \leq k + 1$, then all placements of the new bar add $m$ edges. If $m \geq k+2$, then the addition of the new bar adds between $k+1$ and $\min\{m, 2k+2\}$ edges. In either case the difference between the largest and smallest possible number of additional edges is at most $k+1$. Applying the Azuma-Hoeffding inequality yields the concentration bound.

To find the expected number of edges, we apply linearity of expectation. Suppose $m$ bars have been added so far. If $m \leq k+1$, then as noted above the expected number of edges associated with the new bar is $m$. If $m > k+1$, the expected number of edges between the new bar and bars to its right is $\frac{1}{m+1}\left((m-k)(k+1)+ \sum_{\ell=0}^k \ell\right)$, and by symmetry the total expected number of new edges is twice this number. Summing and simplifying yields the desired bound.
\end{proof}

Felsner and Massow~\cite{felsner2008parameters} established that a semi-bar $k$-visibility graph on $n \geq 2k + 2$ vertices has at most $(k+1)(2n - 2k - 3)$ edges. It is easy to see that the minimum number of edges is of order $(k+1)n$. Theorem~\ref{thm:expectedsemibar} therefore implies that expected number of edges in a random semi-bar $k$-visibility graph has the same first-order behavior as the maximum number of edges for large $n$, and moreover that the fluctuations about the expected value are significantly smaller, of order $O((k+1)\sqrt n)$.

We also consider random semi-arc graphs. We obtain an analogous model to the one considered above for semi-bar graphs by drawing one endpoint of each arc uniformly from $(0, 2 \pi)$.
The resulting graphs have two types of visibilities, depending on whether they pass through the center of the circle. It is clear that those not through the center have the structure of a semi-bar $k$-visibility graph, so the expected number of such edges is given by Theorem~\ref{thm:expectedsemibar}. The following theorem shows that, in expectation, the number of visibilities through the center is of strictly lower order.

\begin{theorem}
Let $G$ be a random semi-arc $k$-visibility graph with $n$ vertices, and let $C$ be the number of edges corresponding to visibilities through the center. Then $\mathbb E [C] \leq \frac{(k+1)(k+2)}{2}\log(n)+O(1)$ where $O(1)$ is independent of $n$.
\end{theorem}
\begin{proof}
Notice that it suffices to consider the case where all endpoints are radially distinct. Define $S_i$ be the set of arcs with exactly $i-1$ longer arcs of smaller radius. Let $E_i$ be $\mathbb E \left[|S_i|\right]$. Label the arcs $a_1, \dots, a_n$ in order of increasing radius. The arc $a_j$ is in $S_i$ if it is the $i$th longest among $a_1, \dots, a_j$. This occurs with probability $1/j$. We therefore obtain $E_i = \sum_{j=i}^n \frac 1 j$ for $i \leq n$.

Now, if an arc is a member of $S_i$ then it is involved in at most $k+1-i$ visibilities through the center. (If $i \geq k+1$, it has no such visibilities). We obtain
\begin{equation*}
\mathbb C \leq \sum_{i=0}^k (k+1 - i) E_i = \binom{k+2}{2} \log n + O(1)\,,
\end{equation*}
as desired.
\end{proof}
\section{Conclusion and open questions}
This work makes progress towards a full understanding of arc and semi-arc $k$-visibility, and in particular gives the first full characterization of the family of semi-arc visibility graphs. It leaves open several questions.
\begin{itemize}
\item Corollary~\ref{cor:tightarcedge} shows that the bound of Theorem~\ref{thm:arcbound} is tight for $k=0$. Is this bound tight for general $k$?
\item Theorems~\ref{thm:semiarck_bound} and~\ref{thm:3k4} give tight bounds on the maximum number of edges in a semi-arc $k$-visibility graphs with $n$ vertices when $n \leq 3k + 4$ or $n \geq 5k + 5$. What is the maximum number of edges for other values of $n$? More concretely, is $K_8$ a semi-arc $1$-visibility graph?
\item Theorem~\ref{thm:ajbi} establishes that $A_j \not \subseteq B_i$ for all $i$ and $j$. Can $A_j \subseteq A_i$ for $i\neq{j}$?
\item What is the maximum possible thickness for arc and semi-arc $k$-visibility graphs? When $k=0$, we note that the bound we obtain for semi-arc visibility graphs is obviously tight, but that the bound for arc visibility graphs is not. Since $K_5$ is an arc-visibility graph, the maximum thickness is at least $2$, and Corollary~\ref{cor:arcthick} shows that it is at most $3$. We conjecture the the former bound is in fact correct.
\end{itemize}
\section{Acknowledgements}
The authors gratefully acknowledge the support of MIT PRIMES. The first author would also like to express his gratitude to his research teachers Richard Kurtz, Jeanette Collette, Dr. Lorraine Solomon, Dr. Daniel Kramer, and especially the second author for their guidance. The second author would like to acknowledge the generous support of the NSF Graduate Research Fellowship under Grant No.~DGE-1122374.
\bibliographystyle{amsplain}
\bibliography{SawWee15}
\end{document}